\documentclass[14pt,a4paper]{extarticle}
\usepackage[T1,T2A]{fontenc} 
\usepackage[utf8]{inputenc} 
\usepackage[russian, english]{babel} 
\usepackage{geometry}
\usepackage{amsmath,amssymb,amsfonts,amsthm,mathrsfs, stmaryrd}
\usepackage{bbm}  % for \mathbbm 1 or ind
\usepackage{bbding}
\usepackage{graphicx,graphics,bm,array}

\usepackage{hyperref} % для ссылок

\numberwithin{equation}{section}

\theoremstyle{plain}
\newtheorem{thm}{Theorem}[section]
\newtheorem{prop}{Proposition}[section]
\newtheorem{cor}{Corollary}[section]
\newtheorem{lem}{Lemma}[section]

\theoremstyle{definition}
\newtheorem{defn}{Definition}
\newtheorem{ex}{Example}[section]

\theoremstyle{remark}
\newtheorem{rem}{Remark}

\newcommand{\cA}{\mathscr A}
\newcommand{\cB}{\mathscr B}
\newcommand{\cF}{\mathscr F}
\newcommand{\cE}{\mathscr E}

\newcommand{\cH}{\mathscr H}
\newcommand{\cM}{\mathscr M}
\newcommand{\cO}{\mathcal O}
\newcommand{\cP}{\mathscr P}
\newcommand{\cT}{\mathcal T}
\newcommand{\RR}{\mathbb R}
\newcommand{\RRl}{\overline{\mathbb R}}
\newcommand{\QQ}{\mathbb Q}
\newcommand{\NN}{\mathbb N}
\newcommand{\FF}{\mathbb F}
\newcommand{\PP}{\mathsf {P}}
\newcommand{\EE}{\mathsf {E}}
\newcommand{\ind}{\mathbbm {1}}

\title{Properties of a Special Type of Filtration and its Martingale Criteria}
\author{
    Assylliya K. Zhunussova\\
\small $^{1}$Lomonosov Moscow State University, Moscow, Russia \\
    \small \href{mailto:lichka\_archive@mail.ru}{lichka\_archive@mail.ru}\,\Envelope, \url{https://orcid.org/0000-0002-8801-8381}}
\date{}
\begin{document}
\maketitle
\thispagestyle{plain} 


\begin{abstract}
    This article investigates the structural properties of stochastic processes relative to a generalized single jump filtration, extending the framework introduced by A.A. Gushchin (2020) to the case of a non-trivial initial $\sigma$-algebra $\cH$. By leveraging the general theory of processes and optional projection techniques, we establish fundamental measurability criteria for random variables and a complete characterization of stopping times and adapted processes. Furthermore, we derive comprehensive martingale and local martingale criteria, providing necessary and sufficient conditions for the preservation of the martingale property in this extended setting.
\end{abstract}
\section{Introduction}

%The study of stochastic processes associated with a single jump time $\gamma$ dates back to the fundamental works of the French school (see \cite{Chou Meyer}, \cite{Dellacherie}, \cite{Dellacherie Meyer}). While the classical theory provides a robust framework for multivariate point processes, recent applications in credit risk and financial engineering (see \cite{Herdegen Herrmann}) demand more flexible filtration structures.

%This paper extends the framework established in Gushchin \cite{Gushchin} by incorporating a non-trivial initial $\sigma$-algebra $\cH$, which represents the information available at $t=0$. Such an extension is not merely technical; it allows for modeling scenarios where agents possess prior knowledge or baseline signals before the jump occurs.

%The presence of $\cH$ introduces significant complexities in the behavior of local martingales and the structure of optional projections. Our main contribution is a systematic characterization of the filtration properties and the derivation of comprehensive martingale criteria, bridging the gap between abstract process theory and applied jump-modeling.

In martingale theory and its applications, identifying the compensator of a piecewise constant process with a single jump is a classical problem. The foundation for this study was laid in the 1970s by the French school, notably in the seminal works of Dellacherie \cite{Dellacherie} and Chou and Meyer \cite{Chou Meyer}, who established the General Theory of Processes and provided comprehensive formulas for multivariate point processes.

Single-jump processes have since become a cornerstone of credit risk theory and financial modeling. A significant impulse was given by Herdegen and Herrmann \cite{Herdegen Herrmann}, who addressed the construction of processes with pre-specified compensators. Recent developments, such as the framework studied by A.A. Gushchin \cite{Gushchin}, demonstrate that practical applications require more general filtration structures where the "initial information" is non-trivial.

However, the structural complexity of these models often leads to non-standard behavior; for instance, a local martingale may fail to be a true martingale, as shown by Ruf \cite{Ruf}. Necessary and sufficient conditions for the preservation of the martingale property were established in \cite{Gushchin Zhunussova}. 

The primary objective of the present work is to provide a systematic characterization of the fundamental objects of stochastic analysis — including stopping times, progressively measurable processes, and (local) martingales — within a single-jump filtration framework that generalizes \cite{Gushchin} by incorporating a non-trivial initial $\sigma$-algebra $\cH$. To achieve this, we leverage the powerful tools of optional and predictable projections as treated in Dellacherie and Meyer \cite{Dellacherie Meyer} and Protter \cite{Protter}.

\subsection{Notations and Preliminaries}
Throughout this paper, we are considering a probability space $(\Omega,\cF, P)$ and a sub-$\sigma$-algebra $\cH\subset\cF$, representing the initial information available at $t=0$. 
Let $\gamma$ be a random variable (the jump time) with the distribution function $G(t)=P(\gamma<t)$. 
Throughout this paper, we operate under the standing assumption that $\PP(\gamma>0)=1$.
Following \cite{Gushchin}, we define the deterministic value $t_G=\sup\left\{t\in\RR_+\,:\, G(t)<1 \right\}$.  
We denote by $\RRl_+=[0,\infty]$ the extended half-line and by $\cB(\RRl_+)$ its Borel $\sigma$-algebra.

We define the system $(\cF_t)_{t\in \RRl_+}$ as follows:
\begin{equation}\label{def: filtration}
    A\in\cF_t\;\, \iff \;\, A\in\cF\quad \exists C(A)\in \cH\;:\; A\cap \{t<\gamma\}=C(A)\cap\{t<\gamma\}
\end{equation}
for $t<\infty$, and the limit $\sigma$-algebra is defined as:
\begin{equation}\label{def: filtration infty}
    \cF_{\infty}=\sigma \left(\bigcup\limits_{t\geq0}\cF_t\right)=\sigma \left(\bigcup\limits_{q\in\QQ_+}\cF_q\right).
\end{equation}
\begin{prop}
The family $(\cF_t)_{t\in\RR_+}$ defined by \eqref{def: filtration} is a filtration.
\end{prop}
\begin{proof}
First, let us show that for each fixed $t\geq 0$, the collection $\cF_t$ is a $\sigma$-algebra on $\Omega$.
\begin{itemize}
    \item \textbf{Empty set and $\Omega$}: Since $\varnothing\in\cF$ and $\varnothing\cap\{t<\gamma\}=\varnothing\cap\{t<\gamma\}$, where $\varnothing\in\cH$, it follows that $\varnothing\in\cF_t$. Similarly, $\Omega\in\cF_t$ by taking $C(\Omega)=\Omega\in\cH$.
    \item \textbf{Complements}: Let $A\in\cF_t$. Then there exists $C(A)\in\cH$ such that $A\cap\{t<\gamma\}=C(A)\cap\{t<\gamma\}$. Consider the complement $A^C=\Omega\backslash A$. We need to find $C(A^C)\in\cH$. Let $C(A^C)=(C(A))^C=\Omega\backslash C(A)$.
    
    Observe that:
    \begin{align*}
        A^C\cap\{t<\gamma\}&=\{t<\gamma\}\backslash (A\cap\{t<\gamma\})=\\
        &=\{t<\gamma\}\backslash (C(A)\cap\{t<\gamma\})=\\
        &=(C(A))^C\cap\{t<\gamma\}
    \end{align*}
    Since $(C(A))^C\in\cH$, we conclude $A^C\in\cF_t$.
    \item \textbf{Countable Unions}: Let $\{A_n\}_{n\geq 1}$ be a sequence of sets in $\cF_t$. Then for each $n$, there exists $C(A_n)\in\cH$ such that $A_n\cap\{t<\gamma\}=C(A_n)\cap\{t<\gamma\}$. Let $A=\bigcup A_n$ and $C(A)=\bigcup C(A_n)$. Since $\cH$ is a $\sigma$-algebra, $C(A)\in\cH$.
    
    Then:
    \begin{align*}
        \left\{\bigcup\limits_{n=1}^\infty A_n\right\}\cap \{t<\gamma\}&=\bigcup\limits_{n=1}^\infty (A_n \cap \{t<\gamma\})=\\
        &= \bigcup\limits_{n=1}^\infty (C(A_n)\cap \{t<\gamma\})=\\
        &=\left\{\bigcup\limits_{n=1}^\infty C(A_n)\right\}\cap \{t<\gamma\}
    \end{align*}
    Thus, $\bigcup A_n\in \cF_t$.
\end{itemize}
Now, let us prove the monotonicity (the filtration property).

Let $s<t$ and $A\in\cF_s$. Then there exists $C(A)\in\cH$ such that $A\cap \{s<\gamma\}=C(A)\cap \{s<\gamma\}$.

Since $\{t<\gamma\}\subset \{s<\gamma\}$, we intersect both sides of the identity with $\{t<\gamma\}$:
$$
(A\cap\{s<\gamma\})\cap\{t<\gamma\}=(C(A)\cap \{s<\gamma\})\cap \{t<\gamma\}.
$$
This simplifies to $A\cap\{t<\gamma\}=C(A)\cap \{t<\gamma\}$, which implies $A\in\cF_t$.

Therefore, $\cF_s\subset \cF_t$, and $(\cF_t)_{t\geq0}$ is a filtration.
\end{proof}
\begin{lem}\label{lem: F is right-continuous}
    The filtration $\FF=(\cF_t)_{t\geq0}$ is right-continuous, i.e., $\cF_t=\bigcap\limits_{s>t}\cF_s$ for all $t\geq0$.
\end{lem}
\begin{proof}
    To establish the right-continuity of the filtration $\FF=(\cF_t)_{t\geq0}$, we observe that for any $t\geq0$, the inclusion $\cF_t\subseteq\bigcap\limits_{s>t}\cF_s$ holds by monotonicity. \\
    Conversely, let $A\in\bigcap\limits_{s>t}\cF_s$. For each $n\in\NN$, there exists $s_n=t+1/n$ and $C_n\in\cH$ such that $A\cap\{s_n<\gamma\}=C_n\cap \{s_n<\gamma\}$. 
    We define $C:=\limsup\limits_{n\to\infty}C_n\in\cH$, which is formally given by the intersection of the tail unions: $\bigcap\limits_{n=1}^\infty\bigcup\limits_{n=k}^\infty C_n$.
    Note that replacing the limit superior with a simple countable union $\bigcup\limits_{n=1}^{\infty} C_n$ is generally invalid here. Without the assumption that $(C_n)$ is an increasing sequence, the countable union would erroneously accumulate elements from $C_k$ on the sets $\{s_n<\gamma\le s_k\}$ for $k < n$. Since the condition $C_k\cap\{s_k<\gamma\}=A\cap\{s_k<\gamma\}$ imposes no restrictions on $C_k$ outsite the event $\{s_k<\gamma\}$, the set $C_k$ may contain elements not belonging to $A$ on these sets. The limit superior avoids this issue because $\limsup\limits_{n\to\infty}C_n$ depends only on the tail of sequence, thus eliminating these extraneous elements.
    To rigorously prove that $A\cap\{t<\gamma\}=C\cap\{t<\gamma\}$, we use the pointwise convergence of indicator functions. By assumption, for each n, we have: $$\ind_{C_n} \cdot\ind_{\{s_n<\gamma\}}=\ind_A\cdot\ind_{\{s_n<\gamma\}}.$$
    Since $s_n\downarrow t$, the expanding sequence of sets $\{s_n<\gamma\}$ monotonically converges to $\{t<\gamma\}$, ensuring that the sequence of indicators $\ind_{\{s_n < \gamma\}}$ has a regular pointwise limit. Taking the limit superior on both sides and using the product property for bounded sequences yields:
    $$\limsup\limits_{n\to\infty}\left( \ind_{C_n}\cdot\ind_{\{s_n<\gamma\}}\right)=\left(\limsup\limits_{n\to\infty}\ind_{C_n}\right)\cdot\lim\limits_{n\to\infty}\ind_{\{s_n<\gamma\}}.$$
    By definition, $\limsup\limits_{n\to\infty}\ind_{C_n}=\ind_C$. 
    Thus, the left-hand side reduces to $\ind_C\cdot\ind_{\{t<\gamma\}}$, while the right-hand side converges to $\ind_A\cdot\ind_{\{t<\gamma\}}$. 
    Equating the limits gives $C\cap\{t<\gamma\}=A\cap\{t<\gamma\}$, which completes the proof.
\end{proof}

\begin{rem}
    Henceforth, we assume that the filtration $(\cF_t)_{t\in\RR_+}$ is augmented by all $P$-null sets of $\cF$. Together with the results of Lemma \ref{lem: F is right-continuous}, this ensures that $\FF$ satisfies the \textbf{usual conditions} (completeness and right-continuity) in the sense of Dellacherie and Meyer \cite{Dellacherie Meyer}.
    This assumption is technical but crucial, as it guarantees the existence and regularity of optional and predictable projection (see \cite[Sec. II.5]{Protter}).
\end{rem}

\begin{defn}
    A filtration $\FF$ satisfying the usual conditions and constructed via \eqref{def: filtration} is called a \textbf{single jump filtration with initial information} $\cH$.
    If $\cH$ is the completion of the trivial $\sigma$-algebra $\{\varnothing, \Omega\}$, this structure reduces to the classical single jump filtration.
\end{defn}

\section{Structural Properties of the Filtration}

\subsection{Measurability and Adaptedness}
In this subsection, we characterize the measurability of random variables and the adaptedness of stochastic processes within the setting of the single jump filtration with initial information $\cH$.
\begin{prop}[Measurability criterion with respect to $\cF_t$]\label{prop: measurability critetion}
    A random variable $\xi$ is $\cF_t$-measurable if and only if there exists an $\cH$-measurable random variable $\eta$ such that
    $$
    \xi(\omega)=\eta(\omega)\quad \text{for all}\;\omega\in\{t<\gamma\}. 
    $$
\end{prop}
\begin{proof}
\textbf{$(\impliedby)$ Sufficiency}.

Suppose there exists an $\cH$-measurable random variable $\eta$ such that $\xi=\eta$ on the set $\{t<\gamma\}$. To prove that $\xi$ is $\cF_t$-measurable, we must show that for any Borel set $B\in\cB(\RR)$, the pre-image $\{\xi\in B\}$ belongs to $\cF_t$.

By the definition of the filtration $\cF_t$, this requires finding a set $C\in\cH$ such that:
$$
\{\xi\in B\}\cap\{t<\gamma\}=C\cap \{t<\gamma\}.
$$
Since $\xi(\omega)=\eta(\omega)$ for all $\omega\in\{t<\gamma\}$, the following identity holds:
$$
\{\omega: \xi(\omega)\in B\}\cap\{t<\gamma\}=\{\omega: \eta(\omega)\in B\}\cap \{t<\gamma\}.
$$
Let $C=\{\eta\in B\}$. Since $\eta$ is $\cH$-measurable, it follows that $C\in\cH$. Thus, the condition for $\cF_t$-measurability is satisfied for any Borel set $B$. This confirms that $\xi$ is $\cF_t$-measurable.

\textbf{($\implies$) Necessity.}
Fix $t\in\RR_+$. Suppose $\xi$ is $\cF_t$-measurable. By the definition of $\cF_t$, for any $r\in\RR$, the set $\{\xi\leq r\}$ belongs to $\cF_t$. Thus, there exists a set $C_r\in \cH$ such that:
\begin{equation}\label{proof: eq: measurable Ft with ind}
    \ind_{\{\xi\leq r\}}\cdot \ind_{\{t<\gamma\}}=\ind_{C_r}\cdot \ind_{\{t<\gamma\}}.
\end{equation}
Taking the conditional expectation with respect to $\cH$ on both sides of \eqref{proof: eq: measurable Ft with ind}, we obtain
$$
\EE\left(\ind_{\{\xi\leq r\}}\cdot \ind_{\{t<\gamma\}}|\cH\right)=\ind_{C_r}\cdot \EE\left(\ind_{\{t<\gamma\}}|\cH\right).
$$
On the set where $\EE(\ind_{\{t<\gamma\}}|\cH)>0$, we define the $\cH$-measurable function:
$$
\varphi(r,\omega)=\frac{\EE(\ind_{\{\xi\leq r\}}\cdot \ind_{\{t<\gamma\}}|\cH)}{\EE(\ind_{\{t<\gamma\}}|\cH)}.
$$
By construction, $\varphi(r,\omega)$ coincides with $\ind_{C_r}$ almost surely on $\{t<\gamma\}$. Since $\ind_{C_r}$ only takes values $0$ and $1$, the function $\varphi(r,\omega)$ is almost surely non-decreasing in $r$ and takes values in $[0,1]$.

To rigorously pass from this family of conditional expectations to a single random variable, we must ensure right-continuity and monotonicity simultaneously for all $r$ outside a single null set. We define the right-continuous modification by taking the infimum over the countable dense subset of rational numbers:
$$\widetilde{\varphi}(r, \omega) := \inf_{q>r, q\in\mathbb{Q}} \varphi(q, \omega).$$
By the properties of the infimum over a dense set, the map $r\mapsto\widetilde{\varphi}(r, \omega)$ is inherently right-continuous and non-decreasing for almost every $\omega$. 
According to classical results on regular conditional distributions (see, e.g., Shiryaev \cite[Chapter II, \S 7]{Shiryaev1996Probability}), this guarantees the existence of an $\cH$-measurable random variable $\eta$ such that $\ind_{\{\eta\le r\}}=\widetilde{\varphi}(r,\omega)$ almost surely.

Finally, we verify the identity:
\begin{align*}
    \ind_{\{\eta\leq r\}}\cdot \ind_{\{t<\gamma\}} &= \frac{\EE\left( \ind_{\{\xi\leq r\}} \cdot \ind_{\{t<\gamma\}}|\cH\right)}{\EE\left(  \ind_{\{t<\gamma\}}|\cH\right)} \cdot \ind_{\{t<\gamma\}}=\\
    &=  \frac{\EE\left( \ind_{C_r} \cdot \ind_{\{t<\gamma\}}|\cH\right)}{\EE\left(  \ind_{\{t<\gamma\}}|\cH\right)} \cdot \ind_{\{t<\gamma\}}=\\
    &= \ind_{C_r} \cdot  \frac{\EE\left( \ind_{\{t<\gamma\}}|\cH\right)}{\EE\left(  \ind_{\{t<\gamma\}}|\cH\right)} \cdot \ind_{\{t<\gamma\}}=\\
    &= \ind_{C_r}  \cdot \ind_{\{t<\gamma\}}=\\
    &= \ind_{\{\xi\leq r\}}  \cdot \ind_{\{t<\gamma\}}
\end{align*}
Since their indicators coincide for all $r$, it follows that $\xi=\eta$ on the set $\{t<\gamma\}$.
\end{proof}

\begin{cor}[Construction of $\cF_t$-measurable versions.]
    Let $\xi$ be an $\cF$-measurable random variable. Then the random variable $\eta$ define by
    $$
    \eta=\EE[\xi|\cH] \ind_{\{t<\gamma\}}+\xi \ind_{\{t\geq\gamma\}}
    $$
    is $\cF_t$-measurable.
\end{cor}
\begin{proof}
    For any $s$, consider the level set $\{\eta<s\}$. Its intersection with the set $\{t<\gamma\}$ satisfies: 
    $$\{\eta < s\} \cap \{t < \gamma\} = \{\EE(\xi|\cH ) < s\} \cap \{t < \gamma\}.$$
    Since $\EE[\xi|\cH]$ is $\cH$-measurable, it follows that $\{\EE[\xi|\cH]<s\}\in\cH$. By \textbf{Proposition} \ref {prop: measurability critetion}, this condition is sufficient to guarantee that $\eta$ is $\cF_t$-measurable. 
\end{proof}

\begin{lem}[Measurability criterion with respect to $\cF_\infty$]
    A random variable $\xi$ is $\cF_\infty$-measurable if and only if there exists an $\cH$-measurable random variable $\eta$ such that
    $$
    \xi(\omega)=\eta(\omega)\quad \text{for all}\;\omega\in\{\gamma=\infty\}. 
    $$
\end{lem}
\begin{proof}
    \textbf{$(\impliedby)$ Sufficiency}.
    
    Suppose there exists an $\cH$-measurable random variable $\eta$ such that $\xi=\eta$ on $\{\gamma=\infty\}$. 
    To establish that the arbitrary $\cF$-measurable random variable $\xi$ is $\cF_\infty$-measurable, we first show that the trace of $\cF_\infty$ on $\{\gamma<\infty\}$ strictly equals the trace of $\cF$ on $\{\gamma<\infty\}$.

    Let $B\in\cF$. For any $n\in\NN$, the intersection $B\cap\{\gamma<\infty\}$ belongs to $\cF_n$, and consequently to the limit $\sigma$-algebra $\cF_\infty\supset\cF_n$.
    Expressing the trace over the finite jump domain as a countable union yields:
    $$
    B\cap \{\gamma<\infty\}=\bigcup\limits_{n=1}^\infty (B\cap\{\gamma\leq n\})\in\cF_\infty.
    $$
    This ensures that the restriction $\xi\ind_{\{\gamma<\infty\}}$ is $\cF_\infty$-measurable.

    On the complementary set $\{\gamma=\infty\}$, the variable $\xi$ coincides with $\eta$ by hypothesis.
    Since $\cH\subset\cF_\infty$, the restriction $\eta\ind_{\{\gamma=\infty\}}$ is inherently $\cF_\infty$-measurable.
    Thus, the decomposition $\xi=\xi\ind_{\{\gamma<\infty\}}+\eta\ind_{\{\gamma=\infty\}}$ guarantees the strict $\cF_\infty$-measurability of the random variable $\xi$.
    
    \textbf{($\implies$) Necessity.}
    
    Let $\cM$ be the collection of all bounded $\cF_{\infty}$-measurable random variables $\xi$ such that there exists an $\cH$-measurable $\eta$ with $\xi=\eta$ on $\{\gamma=\infty\}$. If $0\leq \xi_n\uparrow\xi$ where $\xi_n\in\cM$ and $\xi$ is bounded, then for each $n$ there is an $\cH$-measurable $\eta_n$ such that $\xi_n=\eta_n$ on $\{\gamma=\infty\}$. The limit $\eta=\limsup\eta_n$ is $\cH$-measurable and $\xi=\eta$ on $\{\gamma=\infty\}$, so $\xi\in\cM$.
    Let $\cA=\bigcup\limits_{q\in\QQ_+}\cF_q$. For any indicator $\ind_A$ with $A\in\cA$, Proposition \ref{prop: measurability critetion} guarantees the existence of $\ind_H$ ($H\in\cH$) coinciding with it on $\{q<\gamma\}\supset\{\gamma=\infty\}$. Thus, $\cA\in\cM$.
    Since $\cM$ is a monotone class containing the constant $1$ and the indicators of the generating algebra $\cA$, the Functional Monotone Class Theorem implies that $\cM$ contains all bounded $\cF_{\infty}$-measurable functions. 
    Extension to the unbounded case follows by a standard truncation argument.
\end{proof}

\begin{prop}[Criterion for Adaptedness]\label{prop: adapted process}
    A stochastic process $X = (X_t)_{t\in\RR_+}$ is a process adapted to the filtration $\FF=(\cF_t)_{t\in\RR_+}$ if and only if for each $t\geq0$, the random variable $X_t$ coincides with some $\cH$-measurable random variable $F(t)$ on the set $\{t<\gamma\}$.
\end{prop}
\begin{proof}
    If $X$ is an adapted process, then for each $t\geq0$, the random variable $X_t$ is $\cF_t$-measurable. By the \textbf{measurability criterion} (\textbf{Proposition} \ref{prop: measurability critetion}), there exists an $\cH$-measurable random variable $F(t)$ such that $X_t=F(t)$ on the set $\{t<\gamma\}$. The converse follows directly from the sufficiency part of \textbf{Proposition} \ref{prop: measurability critetion} applied to each $X_t$.
\end{proof}

 \subsection{Characterization of Stopping Times}
\begin{prop}[Stopping Time Criterion]\label{prop: stopping time critetion} 
    A random variable $T$ is a stopping time with respect to $\FF=(\cF_t)_{t\in\RR_+}$ if and only if the following property holds: if the set $\{T<\gamma\}$ is non-empty, then there exists an $\cH$-measurable stopping time $S$ such that
  \begin{equation}\label{eq: T}
        \{T<\gamma\}=\{T=S<\gamma\}=\{S<\gamma\}.
    \end{equation}
\end{prop}
\begin{proof}
\textbf{($\implies$) Necessity.}
Let $T$ be a stopping time with respect to the filtration $\FF=(\cF_t)_{t\geq0}$, meaning $\{T\leq t\}\in\cF_t$ for all $t\geq0$. 
Since $\FF$ is a filtration, for any $s\leq t$, we have $\{T\leq s\}\in\cF_t$.

By the measurability criterion established in \textbf{Proposition} \ref{prop: measurability critetion}, for each $s\leq t$, there exists an $\cH$-measurable set $C_{s,t}$  such that:
$$
\{T\leq s\}\cap \{t<\gamma\}=C_{s,t}\cap \{t<\gamma\},
$$

Following the construction of the regular conditional distribution detailed in the proof of Proposition \ref{prop: measurability critetion}, there exists an $\cH$-measurable random variable $S$ such that $T=S$ on the set $\{T\le t<\gamma\}$. While $T$ and $S$ coincide on the event where $T$ is bounded by $t$ within $\{t<\gamma\}$, these functions may differ on the set $\{T>t\}\cap\{t<\gamma\}$.

Specifically, for any $t\geq0$, we have the identity:
$$
\{T\leq t<\gamma\}=\{T=S\leq t<\gamma\}=\{S\leq t<\gamma\}.
$$
Since this holds for all $t\geq0$, we can consider a countable dense subset of $\RR_+$ (e.g., the rational numbers $\QQ_+$). Taking the union over all $q\in\QQ_+$, we derive the identity \ref{eq: T}.

\textbf{$(\impliedby)$ Sufficiency}.
Conversely, suppose $T$ is a random variable for which there exists an $\cH$-measurable stopping time $S$ satisfying \eqref{eq: T} implies that
$$
\{T\leq t<\gamma\}\subset \{T<\gamma\}\subset \{T=S\}\;\; and \;\;\{S\leq t<\gamma\}\subset \{S<\gamma\}\subset \{T=S\}.
$$
Then, for any $t\geq0$, we obtain:
\begin{align*}
    \{T\leq t\}\cap \{t<\gamma\}&= \{T=S\}\cap \{T\leq t<\gamma\}=\\
    &= \{T=S\}\cap \{S\leq t< \gamma\}=\\
    &= \{S\leq t<\gamma\} = \{S\leq t\}\cap \{t<\gamma\}.
\end{align*}
Since $S$ is an $\cH$-measurable stopping time, the set $\{S\leq t\}$ belongs to $\cH$.
Assuming the non-trivial case $\{T<\gamma\}\neq \varnothing$, the established identity and the definition of the filtration directly imply that $\{T\leq t\}\in \cF_t$ for all $t\geq0$.
(The complementary case $\{T<\gamma\}=\varnothing$ holds trivially, as discussed in Remark \ref{rem: [T<g]=0}).
\end{proof}
\begin{rem}\label{rem: [T<g]=0}
    If the assumption that $\{T<\gamma\}$ is nonempty relaxed, the case $\{T<\gamma\}=\varnothing$ is trivial, as it immediately implies $\{T\leq t\}\cap \{t<\gamma\}=\varnothing\in\cF_t$.
\end{rem}

\begin{lem}[Stopping Time Criterion]\label{lem: stop time criterion}
    Let $(T_n)_{n\geq1}$ be an increasing sequence of stopping times with respect to $\FF$ such that $T_n\to\infty$ almost surely as $n\to\infty$. 
    Assume that $P(\gamma<\infty)=1$ (i.e., $t_G<\infty)$ and $P(T_n<\gamma)>0$ for all $n\geq1$.
    Then there exists an increasing sequence of stopping times $(R_n)_{n\geq1}$ with respect to $\FF$ such that:
    \begin{itemize}
        \item $R_n\leq t_G$ almost surely for all $n\geq1$;
        \item $R_n\to t_G$ almost surely as $n\to\infty$;
        \item For each $n$, the identity holds on the set before the jump:
        \begin{equation}\label{eq: T_n = R_n <gamma}
            \{T_n<\gamma\}=\{T_n=R_n<\gamma\}=\{R_n<\gamma\}.
        \end{equation}
    \end{itemize}
\end{lem}
\begin{proof}
    \textbf{Step 1: Construction of the sequence $(R_n)_{n\geq1}$}
    
    We define the sequence of random variables $(R_n)_{n\geq1}$ by the following formula:
    $$
    R_n:=T_n\ind_{\{T_n<\gamma\}}+t_G\ind_{\{T_n\geq \gamma\}}=
    \left\{
        \begin{array}{ll}
        T_n, & \hbox{if\,\, $T_n<\gamma$;} \\
        t_G, & \hbox{if\,\, $T_n\geq\gamma$.} 
        \end{array}
    \right.
    $$
    By definition, $R_n\leq t_G$ almost surely for all $n\geq1$. Moveover, on the set $\{T_n<\gamma\}$, the identity $R_n=T_n$ holds, which directly implies the jump-coincidence identity \eqref{eq: T_n = R_n <gamma}:
    $$
    \{T_n<\gamma\}=\{T_n=R_n<\gamma\}=\{R_n<\gamma\}.
    $$
    \textbf{Step 2: Verification of the stopping time property}
    
    To show that $R_n$ is a stopping time with respect to $\FF$, we must verify that $\{R_n\leq t\}\in\cF_t$ for any $t\geq0$.
    We perform a two-case analysis.

    For $t<t_G$, we analyze the intersection $\{R_n\leq t\}\cap\{t<\gamma\}$. 
    The condition $R_n\leq t<\gamma\leq t_G$ strictly forces $R_n<t_G$, which by definition implies $R_n=T_n<\gamma$.
    Consequently, on this intersection, the random variables strictly coincide, yielding the identity:
    $$
    \{R_n\leq t\}\cap \{t<\gamma\}=\{T_n\leq t\}\cap \{t<\gamma\}.
    $$
    Since $T_n$ is an $\cF_t$-stopping time, $\{T_n\leq t\}\in\cF_t$, ensuring that $\{R_n\leq t\}\in\cF_t$ for $t<t_G$.

    For $t\geq t_G$, the event $\{t<\gamma\}$ is contained in $\{t_G<\gamma\}$, which is a $\PP$-null set since $\PP(\gamma\leq t_G)=1$.
    Because the filtration $\FF$ is augmented with all $\PP$-null sets, the intersection inherently belongs to $\cF_t$,confirming that $R_n$ is a valid stopping time.
    
    \textbf{Step 3: Convergence to $t_G$}
    
    Since the sequence of stopping times $(T_n)_{n\geq1}$ is monotonically increasing and $T_n\to\infty$ almost surely, the condition $T_n\geq \gamma$ inherently guarantees $T_{n+1}\geq\gamma$. 
    Consequently, the sequence of events $A_n=\{R_n=t_G\}=\{T_n\geq\gamma\}$ is monotonically increasing (i.e. $A_n\subset A_{n+1}$). 
    This monotonicity explicitly justifies the application of the continuity of measure from below:
    \begin{align*}
        P(\underset{n\to\infty}{\lim} R_n=t_G)&=P\left(\bigcup_{n=1}^\infty \{R_n=t_G\}\right)=\\
        &=\underset{n\to\infty}{\lim} P(R_n=t_G)= \underset{n\to\infty}{\lim} P(T_n\geq \gamma )=\\
        &=P(\gamma<\infty)=1.
    \end{align*}
    Thus, $R_n \uparrow t_G$ almost surely as $n\to\infty$.
\end{proof}
\begin{rem}\label{rem: stop time criterion}
    The introduction of the bounded stopping times $R_n\leq t_G$ in Lemma \ref{lem: stop time criterion} is a crucial methodological step. 
    In the general theory of stochastic processes, bounding the time horizon is often the necessary prerequisite for applying Lebesgue's Dominated Convergence Theorem.
    Specifically, restricting a process with regular trajectories (such as càdlàg paths) to a compact time interval $[0,t_G]$ guarantees the existence of an integrable majorant. 
    This ultimately allows for the legal interchange of the limit and the expectation $\lim\limits_{n\to\infty} \EE[X_{t\wedge R_n}]=\EE[\lim\limits_{n\to\infty} X_{t\wedge R_n}]$, preventing the "loss of mass" phenomenon in subsequent martingale proofs.
\end{rem}

\subsection{Predictable and Progressively Measurable Processes}
In this section, we characterize the class of processes that are compatible with the information flow $\FF$ \eqref{def: filtration}; the measurability of a process is largely determined by its behavior prior to the jump $\gamma$.

Throughout this study, we assume that the filtration $\FF$ satisfies the \textbf{usual conditions} (i.e., it is complete and right-continuous). This assumption ensures the equivalence between the concepts of progressive measurability and \textbf{optionality} within the general theory of processes (see Theorem 4 in Section II.5 of \cite{Protter} or VI.43 in \cite{Dellacherie Meyer}).
\begin{prop}[Progressive Measurability Criterion]\label{prop: progressive process}
    A stochastic  process $X$ is progressively measurable with respect to the filtration $\FF$ if and only if there exists a $\cB(\RR_+)\otimes \cH$-measurable function $C(t,\omega)$ such that $X_t=C(t,\omega)$ on the set $\{t<\gamma\}$.
\end{prop}
\begin{proof}
    \textbf{($\implies$) Necessity.}
    Assume $X=(X_t)_{t\geq0}$ to be an $\FF$-progressively measurable process.
    As noted above, under the usual conditions, $X$ is also optional, i.e., measurable with respect to the optional $\sigma$-algebra $\cO(\FF)$.
    
    We establish the existence of the jointly measurable representation $C(t,\omega)$ strictly via the trace properties of $\cO(\FF)$. Given the specific structure of the single jump filtration $\FF$ defined in \eqref{def: filtration}, the trace of $\cF_t$ on $\{t<\gamma\}$ exactly coincides with the trace of the initial $\sigma$-algebra $\cH$ on $\{t<\gamma\}$ for all $t\geq0$. 
    
    Consequently, by the standard properties of generated $\sigma$-algebras (specifically, the identity $\sigma(\cE) \cap \Gamma = \sigma(\cE \cap \Gamma)$ for traces) and the structure of optional processes, the trace of $\cO(\FF)$ on the stochastic interval $\llbracket0,\gamma\llbracket$ identically coincides with the trace of the product $\sigma$-algebra $\cB(\RR_+)\otimes\cH$.

    Since the process $X$ is $\cO(\FF)$-measurable, this coincidence of traces implies the existence of a strictly jointly $\cB(\RR_+)\otimes\cH$-measurable process $C(t,\omega)$ such that $X=C$ on the set $\llbracket 0,\gamma\llbracket$.
    
    \textbf{$(\impliedby)$ Sufficiency}.

    Suppose such a $\cB(\RR_+)\otimes\cH$-measurable function $C(t,\omega)$ exists. 
    We must verify the joint $\cB([0,t])\otimes\cF_t$-measurability of the map $(s,\omega)\to X_s(\omega)$ on the product space $[0,t]\times \Omega$ for any fixed $t\geq 0$.

    The domain $[0,t]\times\Omega$ partitions into two measurable sets: $A=\{(s,\omega): s<\gamma(\omega)\}$ and $B=\{(s,\omega): s\geq\gamma(\omega)\}$.

    On the pre-jump set $A$, we have $X_s(\omega)=C(s,\omega)$ by hypothesis.
    Since the function $C$ is strictly $\cB(\RR_+)\otimes\cH$-measurable and $\cH\subset\cF_t$, its restriction to $A$ is inherently jointly measurable.

    On the post-jump set $B$, the process evaluates to its terminal state: $X_s(\omega)=L(\omega)$ for all $s\geq\gamma(\omega)$.
    Here, the restriction of the process $X$ coincides exactly with the composition of the $\cF_t$-measurable random variable $L$ (since $\gamma\leq s \leq t$) with the measurable section map $(s,\omega)\to\omega$.
    This canonical composition guarantees the required joint $\cB([0,t])\otimes \cF_t$-measurability on $B$.

    Since its restrictions to both $A$ and $B$ are jointly measurable, the stochastic process $X$ is progressively measurable.
\end{proof}
\begin{rem}
    The characterization of progressive measurability in Proposition \ref{prop: progressive process} naturally dictates the structure of stopped processes within the filtration $\FF$. 
    Specifically, if a progressively measurable process $M$ is stopped at the jump time $\gamma$, its pre-jump trajectory is entirely governed by the $\cB(\RR_+)\otimes \cH$-measurable function $F(t):= C(t,\omega)$ on the set $\{t<\gamma\}$, while its terminal state is captured by an $\cF_{\gamma}$-measurable random variable $L:=M_\gamma$ on the set $\{t\geq \gamma\}$. 
    This structural dichotomy rigorously justifies the canonical representation $M_t=F(t)\ind_{\{t<\gamma\}}+L\ind_{\{t\geq \gamma\}}$, which serves as the foundational object for the martingale analysis in the subsequent sections.
\end{rem}

\begin{prop}[Predictable criterion]
    A stochastic process $Y=(Y_t)_{t\geq 0}$ is predictable if and only if there exists a $\cB(\RR_+)\otimes\cH$-measurable function $C(t,\omega)$ such that for all $t\geq0$, $Y_t=C(t)$ on the set $\{t\leq \gamma\}$.
\end{prop}
\begin{proof}
    \textbf{($\implies$) Necessity.} By definition, the predictable $\sigma$-algebra $\cP$ is generated by the class of all left-continuous adapted processes.
    
    Let $Z$ be such a left-continuous adapted process. For any fixed $t$, the random variable $Z_t$ coincides with some $\cH$-measurable variable on the set $\{t\leq \gamma\}$. 
    
    Because $Z$ has left-continuous paths, its restriction to the predictable stochastic interval $[0,\gamma]$ is indistinguishable from a process $C(t,\omega)$ that is a left-continuous in $t$ and $\cH$-measurable in $\omega$. Any such process is inherently $\cB(\RR_+)\otimes \cH$-measurable.
    
    By the Functional Monotone Class Theorem, this structural measurability property extends from left-continuous processes to all $\cP$-measurable (predictable) processes. Thus, for any predictable process $Y$, the required $\cB(\RR_+)\otimes \cH$-measurable process $C(t,\omega)$ exists.
    
    \textbf{$(\impliedby)$ Sufficiency}. Assume such a $\cB(\RR_+)\otimes \cH$-measurable process $C(t,\omega)$ exists. Since $\cH\subset\cF_0$, the process $C$ is predictable.
    
    The stochastic interval $[0, \gamma] = \{(t,\omega) : t \le \gamma(\omega)\}$ is predictably measurable because its indicator process $\ind_{\{t\leq \gamma\}}$ is left-continuous and adapted.
    
    The trajectory of $Y$ prior to and including the jump can be written as $Y_t\ind_{\{t\leq \gamma\}}=C(t,\omega)\ind_{\{t\leq \gamma\}}$. Being the product of two predictable processes, it is predictable itself.
\end{proof}

\begin{ex}[A Predictable Process]
    Consider a financial model where the initial information $\cH$ contains a client's credit rating, modeled by an $\cH$-measurable random variable $\eta$. The discounted value process before default (the jump $\gamma$) can be defined as $Y_t=\eta e^{-rt}\ind_{\{t<\gamma\}}$. Here, the function $C(t,\omega)=\eta(\omega) e^{-rt}$ is explicitly $\cB(\RR_+)\otimes\cH$-measurable and left-continuous, making $Y$ a predictable process.
\end{ex}

\begin{ex}[A Progressively Measurable, but Not Predictable Process]
    Consider the single jump process itself, $X_t=\ind_{\{\gamma\leq t\}}$. Since the filtration is right-continuous and adapted, this right-continuous process is progressively measurable. However, assuming the jump time $\gamma$ is totally inaccessible (e.g., its distribution function $G$ is continuous), $X$ is not predictable, because its jump cannot be foretold by the strict past $\cF_{t-}$.
\end{ex}
\section{Criterion Martingale.}
Let us establish the conditions under which a process is a martingale.

\begin{thm}[Martingale Criterion]\label{thm: martingale criterion}
    Let $\cT=\{t\in\RR_+:\PP(\gamma\geq t)>0\}$. 
    Let $F=(F_t)_{t\in\cT}$ be an $\cH$-measurable process with càdlàg paths, and let $L$ be an arbitrary random variable.
    
    Consider the process $M=(M_t)_{t\in\RR_+}$ defined by the equality:
    \begin{equation}\label{eq: M_t=F(t)1[t<g]+L1[t>=g]}
       M_t =F(t)\ind_{\{t<\gamma\} }+L\ind_{\{t\geq \gamma\}}. 
    \end{equation}   
    Then the following statements are equivalent:
    \begin{enumerate}
        \item \label{thm: martingale criterion: martingale}
        The process $(M_t)_{t\in\cT}$ is a martingale.
        
        \item \label{thm: martingale criterion: 2 eq}
        For all $t\in\cT$, the integrability condition  \begin{equation}\label{eq: E(|Mt|)<8}
            \EE(|M_t|)<\infty
        \end{equation}
        holds, and
        \begin{equation}\label{eq: E(Mt-M0|H)=0}
            \EE(M_t-M_0|\cH)=0\quad\text{a.s.}
        \end{equation}
    \end{enumerate}
\end{thm}
\begin{proof}
    \textbf{($\implies$) Necessity.}
    If $(M_t)_{t\in\cT}$ is a martingale, the integrability condition \eqref{eq: E(|Mt|)<8} is 
    trivially satisfied.
    
    Since $\cH\subset\cF_0$, the tower property of conditional expectation \eqref{eq: E(Mt-M0|H)=0} yields: 
    $$
    \EE(M_t-M_0|\cH)=\EE\Big(\EE(M_t-M_0|\cF_0)\Big|\cH\Big)=\EE(0|\cH)=0.
    $$
    \textbf{$(\impliedby)$ Sufficiency}.
    Conversely, assume statement \ref{thm: martingale criterion: 2 eq} holds. 
    By construction, $M_t$ is right-continuous, adapted (due to the adaptedness criterion in Proposition \ref{prop: adapted process}), and integrable \eqref{eq: E(|Mt|)<8}.
    
    It remains to prove the martingale property: $\EE[M_t-M_s|\cF_s]=0$ a.s. for any $s,t\in\cT$ such that $s\leq t$.
    
    On the set $\{s\geq \gamma\}$, we have $M_t=M_s=L$. Therefore, $M_t-M_s=0$ a.s., which implies $\EE[M_t-M_s|\cF_s]=0$ a.s. on $\{s\geq \gamma\}$.
    
    On the set $\{s<\gamma\}$, the random variable $\EE[M_t-M_s|\cF_s]$ is $\cF_s$-measurable. 
    By the measurability criterion with respect to $\cF_s$ (Proposition \ref{prop: measurability critetion}), there exists an $\cH$-measurable variable $\eta$ such that:
    $$
    \EE(M_t-M_s|\cF_s)\ind_{\{s<\gamma\}}=\eta\ind_{\{s<\gamma\}}.
    $$
    Taking the conditional expectation with respect to $\cH$ and multiplying both sides by $\ind_{\{s<\gamma\}}$, we obtain:
    $$
    \EE\left[(M_t-M_s)\ind_{\{s<\gamma\}}\mid\cH\right] \ind_{\{s<\gamma\}}=\eta\PP(s<\gamma|\cH)\ind_{\{s<\gamma\}}.
    $$
    On the set $\{s<\gamma\}$, the conditional probability $P(s<\gamma\mid\cH)$ is strictly positive a.s., which rigorously justifies dividing by it.
    This yields the identity on $\{s<\gamma\}$:
    $$
    \eta \ind_{\{s<\gamma\}}=\frac{\EE\big((M_t-M_s)\ind_{\{s<\gamma\}}\big|\cH\big)}{P(s<\gamma|\cH)} \ind_{\{s<\gamma\}}.
    $$
    Crucially, since $M_t-M_s=0$ on the complementary set $\{s\geq \gamma\}$, the random variable inside the expectation simplifies: $(M_t-M_s)\ind_{\{s<\gamma\}}=M_t-M_s$.
    Therefore, we can rewrite the numerator:
    $$
    \eta\ind_{\{s<\gamma\}}=\frac{\EE[M_t-M_0\mid\cH]-\EE[M_s-M_0\mid\cH]}{P(s<\gamma|\cH)}\ind_{\{s<\gamma\}}.
    $$
    By the assumption in statement \ref{thm: martingale criterion: 2 eq}, both expectations in the numerator are exactly zero. Thus, $\eta\ind_{\{s<\gamma\}}=0$.
    
    Consequently, $\EE[M_t-M_s|\cF_s]\ind_{\{s<\gamma\}}=0$, concluding the proof that $M$ is a martingale on $\cT$.
\end{proof}

\begin{ex} (Compensated Default Process).
    Consider a model with an $\cH$-measurable random variable $\Lambda>0$, which represents the jump intensity.
    Suppose that, conditional on $\cH$, the jump time $\gamma$ follows an exponential distribution with parameter $\Lambda$, meaning $P(\gamma>t\mid \cH)=e^{-\Lambda t}$.
    
    We define the standard compensated jump process $M_t=\ind_{\{\gamma\leq t\}}-\Lambda(t\wedge \gamma)$.
    Let us verify that $M$ is a martingale using Theorem \ref{thm: martingale criterion}.
    
    Prior to the jump $(t<\gamma)$, the process is $M_t=-\Lambda t$. 
    Thus, $F_t=-\Lambda t$, which is an $\cH$-measurable process with continuous (and therefore càdlàg) paths. 
    At and after the jump $(t\geq \gamma)$, the process is $M_t=1-\Lambda\gamma$, so we identify $L=1-\Lambda\gamma$.
    
    To check the martingale criterion, we compute $\EE[M_t-M_0\mid \cH]$.
    Since $M_0=0$, we have:
    $$
    \EE[M_t\mid\cH]=\EE\left[-\Lambda t \ind_{\{t<\gamma\}}+(1-\Lambda\gamma)\ind_{\{t\geq\gamma\}}\mid\cH\right].
    $$
    Using the conditional density $f(s)=\Lambda e^{-\Lambda s}$, the expectation evaluates to:
    $$
    -\Lambda t P(\gamma>t\mid \cH)+ \int_0^t (1-\Lambda s)\Lambda e^{-\Lambda s}ds=0.
    $$
    Since the conditional expectation is identically zero, by Theorem \ref{thm: martingale criterion}, the process $M$ is indeed a martingale on $\cT$.
\end{ex}
\section{Local Martingales.}
In the previous section, we established the necessary and sufficient conditions for a process to be a martingale in a single-jump filtration with initial information $\cH$.

However, in many advanced stochastic models, particularly in mathematical finance, processes may fail to be true martingales due to a lack of global integrability, while still preserving the local martingale property. In this section, we extend our framework to the broader class of local martingales.

To build a rigorous foundation for our main result (Theorem \ref{thm: local martingale criterion}), we first establish two essential technical propositions detailing the behavior of localizing sequences and the pre-jump properties of local martingales.

\begin{prop}\label{prop: M^tau is martingale => M is local martingale}
    Let $M$ be the process defined by Equation \ref{eq: M_t=F(t)1[t<g]+L1[t>=g]}. If $(M_t)_{t\in\cT}$ is a martingale, then $M$ is a local martingale on $\RRl_+$.
\end{prop}
\begin{proof}
    If $t_G =+\infty$ or $t_G < +\infty$ and $t_G\in\cT$, then $M$ is trivially a uniform martingale on $\RRl_+$. 
    Thus, we consider the case when $t_G < +\infty$ and $\cT=[0,t_G)$.
    Let us choose an increasing sequence $t_1 <\dots < t_n <\dots< t_G$,  such that $t_n\to t_G$. We define:
    $$ 
    T_n:=\left\{
    \begin{array}{ll}
    t_n, & \hbox{if $\gamma >t_n$;} \\
    +\infty, & \hbox{otherwise.}
    \end{array}
    \right.
    $$
    Let us show that $T_n$ is a stopping time. For any $t\geq0$, if $t_n\leq t$, then $\{T_n <t\}=\{\gamma>t_n\}\in\cF_{t_n}\subset\cF_t$; otherwise, $\{T_n\leq t\}=\varnothing\in\cF_t$.
    
    Hence, $T_n$ is a stopping time. Since $P(\gamma<t_G)=1$, we have $T_n\uparrow+\infty$ a.s.
    
    Now consider the stopped process $M^{T_n}_t=M_{t\wedge T_n}$.
    
    If $t < t_n$, then $M_{t\wedge T_n}=M_t$.
    
    If $t\geq t_n$ and $\gamma>t_n$, then $M_{t\wedge T_n}=M_{t_n}$.
    
    If $t\geq t$ and $\gamma\leq t_n$, then $M_{t\wedge T_n}=M_t=L=M_{t_n}$.
    
    As a result, we deduced the identity $M_{t\wedge T_n}=M_{t\wedge t_n}$ for all $t\in\RRl_+$.
    
    Note that since $(M_t)_{t\in\cT}$ is a martingale, the deterministically stopped process $(M_{t\wedge t_n})_{t\in\RRl_+}$ is a uniformly integrable martingale. This means that $M^{T_n}$ is a martingale on $\RRl_+$, and therefore, $M$ is a local martingale.
\end{proof}

\begin{prop}\label{prop: gamma<8 & M is local martingale => M^tau martingale}
    Assume $t_G<+\infty$, $\EE[|L|]<+\infty$ and $\EE[\sup\limits_{t\leq t_G}|F(t)|]<+\infty$.
    Let $M$ be the process defined by Equation \ref{eq: M_t=F(t)1[t<g]+L1[t>=g]}.
    A necessary condition for the process $(M_t)_{t\in\cT}$ to be a martingale is that either $\gamma<t_G$ a.s. or $\EE(L|\cH)\ind _{\{\gamma=t_G\}}=F(0)\ind_{ \{\gamma=t_G\}}$.
\end{prop}
\begin{proof}
     Since $M$ is a local martingale, there exists a localizing sequence of stopping times $(T_n)_{n\geq 1}$ such that $T_n\uparrow+\infty$ a.s., and each stopped process $M^{T_n}$ is a martingale.
     If there exists $n$ such that $P(T_n < \gamma) = 0$, then the process $(M_t)_{t\in\cT}$ is a trivially martingale.
     
     Thus, assume that $P(T_n < \gamma) > 0$ for all $n$.
     Applying Lemma \ref{lem: stop time criterion} to the sequence $(T_n)_{n\geq1}$, we construct an increasing sequence of $\FF$-stopping times $(R_n)_{n\geq1}$ such that $R_n\leq t_G$ a.s., and $R_n\uparrow t_G$ a.s. 
     
     %%%%
     To strictly establish that $M^{R_n}$ inherits the martingale property from $M^{T_n}$, we must  demonstrate their global indistinguishability on $[0,\infty)$.
     We analyze their pathwise coincidence by partitioning the sample space into two sets: $\{T_n<\gamma\}$ and $\{T_n\geq\gamma\}$.

     On the set $\{T_n<\gamma\}$, the construction in Lemma \ref{lem: stop time criterion} explicitly sets $R_n=T_n$.
     Consequently, the stopped processes coincide exactly pathwise for all $t\geq0$.

     On the complementary set $\{T_n\geq\gamma\}$, the jump $\gamma$ occurs before or at the stopping time $T_n$.
     Since $R_n=t_G\geq \gamma$ on this set, both stopping times occur after the jump.
     By the fundamental structure of the process $M$, for any $t\geq\gamma$, the process is constant and equal to $L$.
     Therefore, both stopped processes eventually equal $L$ (specifically, $M^{T_n}=L$ for $t\geq\gamma$, and similarly $M^{R_n}=L$ since $R_n\geq \gamma$).
     Since $M^{T_n}$ and $M^{R_n}$ possess càdlàg paths, which are uniquely determined by their values on a dense set of rational times, this eventual equality guarantees their indistinguishability on $[0,\infty)$ outside a single null set.

     This global indistinguishability rigorously justifies the preservation of the martingale property for $M^{R_n}$.
     %%%%
     
     Taking into account Remark \ref{rem: stop time criterion}, the stopped process is bounded by a random variable $Y=\sup\limits_{s\leq t_G}|F(s)|+|L|$.
     Given our initial assumptions that $\EE[\sup\limits_{t\leq t_G} |F(t)|]<+\infty$ and $\EE[|L|]<+\infty$, the random variable $Y$ is an integrable majorant ($\EE[Y]<+\infty$). 
     Consequently, Lebesgue's Dominated Convergence Theorem justifies the interchange of the limit and the expectation:
     $$
     \lim\limits_{n\to\infty}\EE[M_{t\wedge R_n}\mid \cH]=\EE\Big[\lim\limits_{n\to\infty} M_{t\wedge R_n}\mid \cH\Big]
     $$
     and
     $$
     \lim\limits_{n\to\infty}\EE[|M_{t\wedge R_n}|]=\EE\Big[\lim\limits_{n\to\infty} |M_{t\wedge R_n}|\Big]
     $$
     Therefore, $M$ is a martingale on $[0,t_G)$.
     
     We now consider two exhaustive cases: $P(\gamma=t_G)=0$ and $P(\gamma=t_G)>0$. 
     
     If $P(\gamma=t_G)=0$, then $\gamma<t_G$ a.s., which means $\cT=[0,t_G)$. In this case, $(M_t)_{t\in\cT}$ is immediately a martingale. 
     
     If $P(\gamma=t_G)>0$, then domain is $\cT=[0,t_G]$. For the process to be a martingale on the entire set $\cT$, the martingale properties must hold at the boundary point $t_G$.
     
     According to Theorem \ref{thm: martingale criterion}, this necessitates the integrability condition $\EE[|M_{t_G}|]<\infty$ and the centering condition $\EE[M_{t_G}-M_0\mid \cH]=0$ a.s. on the set $\{\gamma=t_G\}$.
     
     On this specific set, the jump occurs exactly at $t_G$, meaning $M_{t_G}=L$ and $M_0=F(0)$.
     Substituting this identity into the criterion directly yields $\EE[L\mid\cH]\ind_{\{\gamma=t_G\}}=F(0)\ind_{\{\gamma=t_G\}}$, concluding the proof.
\end{proof}

\begin{cor}[Local Martingale Criterion in Case  $\gamma<t_G<+\infty$ a.s.]\label{prop: P(gamma=t_G)=0 & M is local martingale => M^tau martingale}
    Assume $t_G<+\infty$, $P(\gamma=t_G)=0$, $\EE[|L|]<+\infty$, and $\EE[\sup\limits_{t\leq t_G} |F(t)|]<+\infty$.
    Let $M$ be the process defined by Equation \ref{eq: M_t=F(t)1[t<g]+L1[t>=g]}.
    Then the following statements are equivalent: 
    \begin{itemize}
        \item The process $M$ is a local martingale on $\RRl_+$.
        \item The process $(M_t)_{t\in[0,t_G)}$ is a martingale.
    \end{itemize}
\end{cor}
\begin{prop}\label{prop: gamma=8 & M is local martingale => M^tau martingale}
    Assume $P(\gamma=+\infty)>0$, $\EE[\sup\limits_{t\in[0,+\infty)} |F(t)|]<+\infty$ and $\EE[|L|]<+\infty$. 
    Let $M$ be a local martingale on $\RRl_+$ defined by $M_t=F(t)\ind_{\{t<\gamma\}}+L\ind_{\{t\geq\gamma\}}$. 
    A necessary condition for the process $(M_t)_{t\in\RRl_+}$ to be a martingale is that the following equality holds on the set $\{\gamma=+\infty\}$: 
    $$
    \EE[L\mid\cH]\ind_{\{\gamma=+\infty\}}=F(0)\ind_{\{\gamma=+\infty\}}
    $$
\end{prop}
\begin{proof}
    Let $M$ be a local martingale on $\RRl_+$. 
    By definition, there exists a localizing sequence of stopping times $(T_n)_{n\geq1}$ such that $T_n\uparrow+\infty$ a.s., and each stopped process $M^{T_n}$ is a true martingale.
    
    To analyze the global martingale property at the terminal point $t=+\infty$, we must account for the entire sample space $\Omega$, which we naturally partition into two disjoint sets: $\{\gamma<+\infty\}$ and $\{\gamma=+\infty\}$.
    
    On the set $\{\gamma<+\infty\}$, the jump occurs in finite time.
    The localizing sequence $(T_n)_{n\geq1}$ eventually captures the terminal state $L$ (analogous to the mechanisms detailed in Proposition \ref{prop: gamma<8 & M is local martingale => M^tau martingale}).
    
    Therefore, to establish the necessary conditions specific to the extended horizon, we restrict our analysis to the complementary set $\{\gamma=+\infty\}$.
    
    Consider the subset $\{\gamma=+\infty\}$. On this set, the jump does not occur in finite time.
    Consequently, the filtration remains constant and equals the initial information ($\cF_t\equiv\cH$ for all $t<+\infty$).
    
    Since any local martingale with respect to a constant filtration is constant in time, the process takes the time-independent value $M_t=F(0)$ for all $t<+\infty$ on this set.
    
    For any fixed $n$, the stopping time $T_n$ is finite almost surely. 
    Therefore, $T_n<\gamma$ on the set $\{\gamma=+\infty\}$.
    
    Thus, the stopped process evaluated at infinity is exactly $M_{+\infty}^{T_n}=M_{T_n}=F(0)$.
    
    This demonstrates that the stopped process $M^{T_n}$ never reaches the terminal state and is entirely independent of the random variable $L$.
    Consequently, the local martingale property of $M$ inherently imposes no restrictions on $L$ at infinity.
    
    In order for the unstopped process $(M_t)_{t\in\RRl_+}$ to be a true martingale globally, it must satisfy the martingale definition directly at the terminal point $t=+\infty$.
    
    This requires $\EE[M_{+\infty}|\cF_0]=M_0$.
    
    Given our initial assumption that $\EE[|L|]<+\infty$, the conditional expectation is well-defined. Noting that $M_{+\infty}=L$ and $M_0=F(0)$ on the set $\{\gamma=+\infty\}$, and since $\cF_0=\cH$, multiplying both sides of the martingale equality by the indicator $\ind_{\{\gamma=+\infty\}}$ directly yields $\EE[L|\cH]\ind_{\{\gamma=+\infty\}}=F(0)\ind_{\{\gamma=+\infty\}}$.
\end{proof}

\begin{thm}[Local Martingale Criterion]\label{thm: local martingale criterion}
    Let $F=(F(t))_{t\in\cT}$ be an $\cH$-measurable process with càdlàg paths and $\EE[\sup\limits_{t\in[0,t_G)} |F(t)|]<+\infty$. Let $L$ be an integrable random variable (i.e., $\EE[|L|]<+\infty$).
    We define the process $M = (M_t)_{t\in\RRl_+} $ by the equality: 
    \begin{equation}
       M_t =F(t)\ind_{\{t<\gamma\} }+L\ind_{\{t\geq \gamma\}}. 
    \end{equation}   
    Then the following are equivalent: 
    \begin{enumerate}
        \item \label{thm: local martingale: martingale} $(M_t)_{t\in\cT}$ is a martingale.
        \item \label{thm: local martingale: local martingale} The process $M$ is local martingale on $\RRl_+$, and the following equality holds on the set $\{\gamma=t_G\}$: 
        $$
        \EE(L|\cH)\ind \{\gamma=t_G\}=F(0)\ind \{\gamma=t_G\}
        $$
    \end{enumerate}
\end{thm}
\begin{proof}
    The underlying structural requirements—specifically, that $F$ is an $\cH$-measurable càdlàg process and the integrability conditions $\EE[\sup\limits_{t \in [0,t_G)} |F(t)|] < +\infty$ and $\EE[|L|] < +\infty$ hold—ensure that the fundamental Martingale Criterion (Theorem \ref{thm: martingale criterion}) is applicable across all subsequent propositions.

    \eqref{thm: local martingale: martingale} $\Longrightarrow$ \eqref{thm: local martingale: local martingale}: If $(M_t)_{t\in\cT}$ is a martingale, Proposition \ref{prop: M^tau is martingale => M is local martingale} directly guarantees that $M$ is a local martingale on $\overline{\mathbb{R}}_+$.
    
    \eqref{thm: local martingale: local martingale} $\Longrightarrow$ \eqref{thm: local martingale: martingale}: Assume $M$ is a local martingale on $\overline{\mathbb{R}}_+$ and the boundary equality holds.
    We demonstrate that $(M_t)_{t\in\cT}$ is a martingale by partitioning the analysis into three exhaustive cases based on the distribution of $\gamma$:
    \begin{itemize}
        \item $\gamma < +\infty$ a.s. and $P(\gamma = t_G) = 0$. The domain is strictly $\cT = [0, t_G)$. According to Corollary \ref{prop: P(gamma=t_G)=0 & M is local martingale => M^tau martingale}, the local martingale property, combined with the integrability of $L$ and the bounded supremum of $F$, is sufficient to conclude that $(M_t)_{t\in\cT}$ is a martingale.
        \item  $\gamma < +\infty$ a.s. and $P(\gamma = t_G) > 0$. The domain is $\cT = [0, t_G]$. Proposition \ref{prop: gamma<8 & M is local martingale => M^tau martingale} dictates that, alongside the integrability conditions, the true martingale property requires the explicit martingale equality at $t_G$. Since this equality is given by our assumption, $(M_t)_{t\in\cT}$ is a martingale.
        \item $P(\gamma = +\infty) > 0$ (i.e., $t_G = +\infty$). The boundary is at infinity. Proposition \ref{prop: gamma=8 & M is local martingale => M^tau martingale} establishes that the local martingale property, supported by the integrability conditions and the specific boundary equality at $+\infty$, ensures the martingale property on the entire extended axis $\overline{\mathbb{R}}_+$.
    \end{itemize}
    Thus, in all possible scenarios, $(M_t)_{t\in\cT}$ is a true martingale. 
\end{proof}

\section{Compliance with Ethical Standards}

\noindent \textbf{Funding:} The author declares that no funds, grants, or other support were received during the preparation of this manuscript.\\
\noindent \textbf{Disclosure of potential conflicts of interest:} The author has no relevant financial or non-financial interests to disclose.\\
\noindent \textbf{Authors Contribution:} Not applicable. The manuscript was written entirely by the sole author. \\
\noindent \textbf{Data Availability Statement:} Data sharing is not applicable to this article as no datasets were generated or analyzed during the current study.

\section{Acknowledgments}
The author wishes to express her sincere gratitude to her scientific advisor, Alexander Alexandrovich Gushchin, for suggesting the problem framework, continuous support, and insightful mathematical discussions that helped shape this work.

\end{document}